\documentclass[12pt,oneside,reqno]{amsart}

\usepackage{amssymb,latexsym,amsxtra,amscd,ifthen}

\usepackage{amsmath}

\usepackage{amsfonts}

\usepackage{verbatim}

\usepackage{dsfont}

\usepackage{mathtools}

\usepackage{enumerate}

\usepackage[capitalise]{cleveref}

\usepackage{color}

\usepackage{endnotes}

\theoremstyle{plain}

\newtheorem{theorem}{Theorem}

\newtheorem{lemma}[theorem]{Lemma}

\newtheorem{corollary}[theorem]{Corollary}

\theoremstyle{definition}

\theoremstyle{remark}

\newenvironment{thmbis}[1]
  {\renewcommand{\thetheorem}{\ref{#1}$'$}%
   \addtocounter{theorem}{-1}%
   \begin{theorem}}
  {\end{theorem}}

\DeclareMathAlphabet{\mathbbold}{U}{bbold}{m}{n}
\def\bb1{\mathbbold{1}}
\def\bbz{\mathbb{Z}}

\def\bbn{\mathbb{N}}

\DeclareMathOperator\ad{ad}

\newcommand{\til}[1]{\widetilde{#1}}

\newtheorem*{remark*}{Remark}

\usepackage{xparse}
\DeclareDocumentCommand{\adx}{ O{2} O{x_1}  }{\ad_{#2}^{[#1]}}

\begin{document}
\title{Algebras and semigroups of locally subexponential growth}
\author{Adel Alahmadi\protect\endnotemark[1], Hamed Alsulami\protect\endnotemark[1], S.K. Jain\protect\endnotemark[1]$^,$\protect\endnotemark[2], Efim Zelmanov\protect\endnotemark[1]$^,$\protect\endnotemark[3]$^{,1}$}
	
	\keywords{growth function, associative algebra, wreath product}
	

	\address{1. To whom correspondence should be addressed\hfill\break
E-mail: EZELMANO@MATH.UCSD.EDU\hfill\break
Author Contributions: A.A., H.A, S. K. J., E. Z. designed research; performed research and wrote the paper. The authors declare no conflict of interest.}
	
	\begin{abstract}
	We prove that a countable dimensional associative algebra (resp. a countable semigroup) of locally subexponential growth is $M_\infty$-embeddable as a left ideal in a finitely generated algebra (resp. semigroup) of subexponential growth. Moreover, we provide bounds for the growth of the finitely generated algebra (resp. semigroup). The proof is based on a new construction of matrix wreath product of algebras.
	\end{abstract}

\maketitle

\section{Introduction}

G. Higman, H. Neumann and B. H. Neumann \cite{6} proved that every countable group embeds in a finitely generated group. The papers \cite{9}, \cite{10}, \cite{11}, \cite{13} show that some important properties can be inherited by these embeddings. In the recent remarkable paper \cite{2}, L. Bartholdi and A. Erschler proved that a countable group of locally subexponential growth embeds in a finitely generated group of subexponential growth.

Following the paper \cite{6}, A. I. Malcev \cite{7} showed that every countable dimensional algebra over a field embeds into a finitely generated algebra.

Let $A$ be an associative algebra over a ground field $F$. Let $X$ be a countable set. Consider the algebra $M_\infty(A)$ of $X \times X$ matrices over $A$ having finitely many nonzero entries. Clearly, there are many ways the algebra $A$ embeds into $M_\infty(A)$.

We say that an algebra $A$ is $M_{\infty}$-embeddable in an algebra $B$ if there exists an embedding $\varphi: M_\infty(A)\rightarrow B$. The algebra $A$ is $M_\infty$-embeddable in $B$ as a (left, right) ideal if the image of $\varphi$ is a (left, right) ideal of $B$.

The construction of a wreath product in \cite{1} implied the following refinement of the theorem of Malcev: every countable dimensional algebra is $M_\infty$-embeddable in a finitely generated algebra as an ideal.

In this paper, we
\begin{enumerate}[(1)]
\item prove the analog of Bartholdi-Erschler theorem for algebras: every countable dimensional associative algebra of locally subexponential growth is $M_\infty$-embeddable in a 2-generated algebra of subexponential growth as a left ideal;
\item provide estimates for the growth of the finitely generated algebra above;
\item consider the case of a countable dimensional algebra of Gelfand-Kirillov dimension $\leq d$ and $M_\infty$-embed it in a 2-generated algebra of Gelfand-Kirillov dimension $\leq d+2$ as a left ideal;
\item establish the similar results for semigroups.
\end{enumerate}

J. Bell, L. Small and A. Smoktunowicz \cite{3} embedded an arbitrary, countable dimensional algebra of Gelfand-Kirillov dimension $d$ in a 2-generated algebra of Gelfand-Kirillov dimension $\leq d+2$.



\section{Definitions and Main Results}

Let $A$ be an associative algebra over a ground field $F$ that is generated by a finite dimensional subspace $V$. Let $V^n$ denote the span of all products $v_1\cdots v_k$, where $v_i\in V$, $k\leq n$. Then $V^1 \subseteq V^2 \subseteq \cdots$ and $\bigcup\limits_{n\geq 1} V^n = A$. The function $g(V,n)=\dim_FV^n$ is called the growth function of $A$.

Let $\bbz$ and $\bbn$ denote the set of integers and the set of positive integers, respectively. Given two functions $f,g:\bbn\rightarrow [1, \infty)$, we say that $f \preceq g$ ($f$ is asymptotically less than or equal to $g$) if there exists a constant $c\in \bbn$ such that $f(n)\leq cg(cn)$ for all $n \in \bbn$. If $f \preceq g$ and $g\preceq f$, then $f$ and $g$ are said to be\underline{asymptotically equivalent}, i.e., $f \sim g$.

We say that a function $f$ is \underline{weakly asymptotically less than or equal to} $g$ if for arbitrary $\alpha > 0$ we have $f\preceq gn^\alpha$ (denoted $f\preceq_w g$).

If $V,W$ are finite dimensional generating subspaces of $A$, then \linebreak $g(V,n)\sim g(W,n)$. We will denote the class of equivalence of $g(V,n)$ as $g_A$.

A function $f: \bbn \rightarrow [1,\infty)$ is said to be \underline{subexponential} if for an arbitrary $\alpha>0$
\[\lim\limits_{n\rightarrow\infty} \dfrac{f(n)}{e^{\alpha n}}=0.\]

For a growth function $f(n)$ of an algebra, it is equivalent to $f(n) \precnapprox e^n$ and to $\lim\limits_{n\rightarrow \infty}\sqrt[n]{f(n)}=1$.

If a function $f(n)$ is subexponential but $n^\alpha \precnapprox f(n)$ for any $\alpha >0$, then $f(n)$ is said to be intermediate. In the seminal paper \cite{5}, R. I. Grigorchuk constructed the first example of a group with an intermediate growth function. Finitely generated associative algebras with intermediate growth functions are more abundant (see \cite{12}).

A not necessarily finitely generated algebra $A$ is of locally subexponential growth if every finitely generated subalgebra $B$ of $A$ has a subexponential growth function.

We say that the growth of $A$ is locally (weakly) bounded by a function $f(n)$ if for an arbitrary finitely generated subalgebra of $A$, its growth function is $\preceq f(n)$ (resp. $\preceq_w f(n)$).

A function $h(n)$ is said to be \underline{superlinear} if $\dfrac{h(n)}{n} \rightarrow \infty$ as $n \rightarrow \infty$.

The main result of this paper is:

\begin{theorem}\label{Theorem1}
Let $f(n)$ be an increasing function. Let $A$ be a countable dimensional associative algebra whose growth is locally weakly bounded by $f(n)$. Let $h(n)$ be a superlinear function. Then the algebra $A$ is $M_\infty$-embeddable as a left ideal in a 2-generated algebra whose growth is weakly bounded by $f(h(n))n^2$.
\end{theorem}

We then use Theorem \ref{Theorem1} to derive an analog of the Bartholdi-Erschler theorem (see \cite{2}).

\begin{theorem}\label{Theorem2}
A countable dimensional associative algebra of locally subexponential growth is $M_\infty$-embeddable in a 2-generated algebra of subexponential growth as a left ideal.
\end{theorem}

A finitely generated algebra $A$ has polynomially bounded growth if there exists $\alpha > 0$ such that $g_A \preceq n^\alpha$. Then
\[GK\dim(A)=\inf\{\alpha>0 | g_A \preceq n^\alpha\}\]
is called the \underline{Gelfand-Kirillov dimension} of $A$. If the growth of $A$ is not polynomially bounded, then we let $GK\dim(A)=\infty$. If the algebra $A$ is not finitely generated then the Gelfand-Kirillov dimension of $A$ is defined as the supremum of Gelfand-Kirillov dimensions of all finitely generated subalgebras of $A$.

J. Bell, L. Small, A. Smoktunowicz \cite{3} proved that every countable dimensional algebra of Gelfand-Kirillov dimension $\leq n$ is embeddable in a 2-generated algebra of Gelfand-Kirillov dimension $\leq n+2$.

We use Theorem \ref{Theorem1} to prove

\begin{theorem}\label{Theorem3}
Every countable dimensional algebra of Gelfand-Kirillov dimension $\leq n$ is $M_\infty$-embeddable in a 2-generated algebra of Gelfand-Kirillov dimension $\leq n+2$ as a left ideal.
\end{theorem}

The proof of Theorem \ref{Theorem1} is based on a new construction of the matrix wreath product $A \wr F[t^{-1},t]$. We view it as an analog of the wreath product of a group $G$ with an infinite cyclic group $\bbz$ that played an essential role in the Bartholdi-Erschler proof \cite{2}.

The construction is similar to that of \cite{1}, though not quite the same.

Analogs of Theorems \ref{Theorem1}, \ref{Theorem2}, \ref{Theorem3} are true also for semigroups. Recall that T. Evans \cite{4} proved that every countable semigroup is embeddable in a finite 2-generated semigroup.

We will formulate analogs of Theorems \ref{Theorem1}, \ref{Theorem2}, \ref{Theorem3} for semigroups for the sake of completeness, omitting some definitions that are similar to those for algebras.

Let $P$ be a semigroup. Consider the Rees type semigroup
\[M_\infty(P)=\bigcup\limits_{i,j\in\bbz}e_{ij}(P),\]
$e_{ij}(a)e_{kq}(b)=\delta_{jk}e_{iq}(ab)$; $a,b\in P$. We say that a semigroup $P$ is $M_\infty$-embeddable in a semigroup $S$ if there is an embedding $\varphi:M_\infty(P)\rightarrow S$. We say that $P$ is $M_\infty$-embeddable in $S$ as a (left) ideal if $\varphi(M_\infty(P))$ is a (left) ideal of $S$.

\begin{thmbis}{Theorem1}\label{Theorem1'}
Let $f(n)$ be an increasing function. Let $P$ be a countable semigroup whose growth is locally weakly bounded by $f(n)$. Let $h(n)$ be a superlinear function. Then the semigroup $P$ is $M_\infty$-embeddable as a left ideal in a finitely generated semigroup whose growth is weakly bounded by $f(h(n))n^2$.
\end{thmbis}

\begin{thmbis}{Theorem2}\label{Theorem2'}
A countable semigroup of locally subexponential growth is $M_\infty$-embeddable in a finitely generated semigroup of subexponential growth as a left ideal.
\end{thmbis}

\begin{thmbis}{Theorem3}\label{Theorem3'}
Every countable semigroup of Gelfand-Kirillov dimension $\leq d$ is $M_\infty$-embeddable in a finitely generated semigroup of Gelfand-Kirillov dimension $\leq d+2$ as a left ideal.
\end{thmbis}

\section{Matrix Wreath Products}

As above, let $\bbz$ be the ring of integers. For an associative $F$-algebra $A$, consider the algebra $\til{M}_\infty(A)$ of infinite $\bbz \times \bbz$ matrices over $A$ having finitely many nonzero entries in each column. The subalgebra of $\til{M}_\infty(A)$ that consists of matrices having finitely many nonzero entries is denoted as $M_\infty(A)$. Clearly, $M_\infty(A)$ is a left ideal of $\til{M}_\infty(A)$.

For an element $a \in A$ and integers $i,j\in\bbz$, let $e_{ij}(a)$ denote the matrix having $a$ in the position $(i,j)$ and zeros everywhere else. For a matrix $X \in \til{M}_\infty(A)$, the entry at the position $(i,j)$ is denoted as $X_{i,j}$.

The vector space $\til{M}_\infty(A)$ is a bimodule over the algebra $F[t^{-1},t]$ via the operations: if $X \in \til{M}_\infty(A)$ then $(t^kX)_{i,j}=X_{i-k,j}$ for all $i,j,k\in \bbz$. In other words left multiplication by $t^k$ moves all rows of $X$ up by $k$ steps. Similarly, $(Xt^k)_{i,j}=X_{i,j+k}$, so multiplication by $t^k$ on the right moves all columns of $X$ left by $k$ steps.

Consider the semidirect sum
\[A \wr F[t^{-1},t]=F[t^{-1},t]+\til{M}_\infty(A)\]
and its subalgebra
\[A\text{ } \bar{\wr} \text{ }F[t^{-1},t]=F[t^{-1},t]+M_\infty(A).\]
These algebras are analogs of the unrestricted and restricted wreath products of groups with $\bbz$.

Let $A$ be a countable dimensional algebra with $1$. We say that a matrix $X \in \til{M}_\infty(A)$ is a \underline{generating matrix} if the entries of $X$ generate $A$ as an algebra.

Let $X \in \til{M}_\infty(A)$ be a generating matrix. Consider the subalgebra of $A \wr F[t^{-1},t]$ generated by $t^{-1},t,e_{00}(1),X,$
\[S=\langle t^{-1}, t, e_{00}(1), X \rangle.\]

\begin{lemma}\label{Lemma1}
The algebra $M_\infty(A)$ is a left ideal of $S$.
\end{lemma}
\begin{proof}
Suppose that entries $X_{i_1,j_1}, X_{i_2,j_2},\cdots$ generate $A$. We have $e_{ij}(1)=t^ie_{00}(1)t^{-j}$ and
\[e_{00}(X_{i_k,j_k})=e_{0i_k}(1)Xe_{j_k0}(1)=e_{00}(1)t^{-i_k}Xt^{j_k}e_{00}(1)\in S.\]
This implies that $e_{00}(A) \subseteq S$ and therefore $e_{ij}(A)=t^ie_{00}(A)t^{-j}\subseteq S$. We proved that $M_\infty(A)=\sum\limits_{i,j\in \bbz}e_{ij}(A)\subseteq S$.

Since $M_\infty(A)$ is a left ideal in the algebra $A \wr F[t^{-1},t]$ the assertion of the lemma follows.
\end{proof}

For a fixed $n\in\bbz$ by $n^{th}$ diagonal we mean all integers pairs $(i,j)$ such that $i-j=n$.

\begin{lemma}\label{Lemma2}
If a generating matrix $X$ has finitely many nonzero diagonals, then $M_\infty(A)$ is a two-sided ideal in $S$.
\end{lemma}
\begin{proof}
If a matrix $X$ has finitely many nonzero diagonals, then $M_\infty(A)X \subseteq M_\infty(A)$, which implies the claim.
\end{proof}

We say that a sequence $c=(a_1, a_2, a_3, \cdots)$ of elements of the algebra $A$ is a generating sequence if the elements $a_1, a_2, \cdots$ generate $A$.

For the sequence $c$, consider the matrix $c_{0,N}=\sum\limits_{j=1}^{\infty}e_{0j}(a_j)\in \til{M}_\infty(A)$. This matrix has elements $a_j$ at the positions $(0,j)$, $j\geq 1$, and zeros everywhere else.

Consider the subalgebra
\[A^{(c)}=\langle t,t^{-1},e_{00}(1), c_{0,N}\rangle\]
of the matrix wreath product $A \wr F[t^{-1},t]$. As shown in Lemma \ref{Lemma1}, the countable dimensional algebra $A$ is $M_\infty$-embeddable in the finitely generated algebra $A^{(c)}$ as a left ideal.

When speaking about algebras $A^{(c)}$ we always consider the generating subspace $V=span(t,t^{-1},e_{00}(1),c_{0,N})$ and denote $g(V,n)=g(n)$.

For a generating sequence $c=(a_1,a_2,\cdots)$, let $W_n$ be the subspace of $A$ spanned by all products $a_{i_1}\cdots a_{i_r}$ such that $i_1+\cdots+i_r\leq n$.

Denote
\[M_{[-n,n]\times[-n,n]}(W_n)=\sum\limits_{-n\leq i,j\leq n} e_{ij}(W_n),\]
\[M_{[-n,n]\times 0}(W_n)=\sum\limits_{i=-n}^n e_{i0}(W_n).\]

\hfill

\begin{lemma}\label{Lemma3}
\begin{enumerate}
\item $e_{00}(W_n)\subseteq V^{2n+1}$;
\item \label{L3I2} \hspace{1cm}

$\begin{array}{r c l}
V^n&\subseteq& M_{[-n,n]\times[-n,n]}(W_n)\\
&&+\sum\limits_{\substack{i\geq 1, -n\leq j\leq n,\\ i+|j|\leq n}}M_{[-n,n]\times0}(W_i)c_{0N}t^j+\sum\limits_{j=-n}^nFt^j.
\end{array}$
\end{enumerate}
\end{lemma}
\begin{proof}
If $i_1+\cdots+i_r\leq n$, then $e_{00}(a_{i_1}\cdots a_{i_r})=c_{0N}t^{i_1}c_{0N}t^{i_r}e_{00}(1) \in V^{2n+1}$, which proves part $(1)$.

Let us start the proof of part $(2)$ with the inclusion
\[e_{00}(1)V^ne_{00}(1)\subseteq e_{00}(W_n).\]

Let $w$ be the product of length $\leq n$ in $t^{-1},t,e_{00}(1),c_{0N}$. If $w$ does not involve $c_{0N}$, then $we_{00}(1)\in \sum\limits_{i=-n}^nFe_{i0}(1)$ and therefore $e_{00}(1)we_{00}(1)\in Fe_{00}(1)$.

Suppose now that $w$ involves $c_{0N}$, $w=w'c_{0N}w''$, the subproduct $w''$ does not involve $c_{0N}$. We have $c_{0N}=e_{00}(1)c_{0N}$. Hence $e_{00}(1)we_{00}(1)=e_{00}(1)w'e_{00}(1)c_{0N}w''e_{00}(1)$. Let $d_1$ be the length of the product $w'$, and let $d_2$ be the length of the product $w''$ with $d_1+d_2 \leq n-1$. By the induction assumption on the length of the product, we have $e_{00}(1)w'e_{00}(1) \in e_{00}(w_{d_1})$. As we have mentioned above $w''e_{00}(1) \in \sum\limits_{i=-d_2}^{d_2}Fe_{i0}(1)$. It is straightforward that
\[c_{0N}(\sum\limits_{i=-d_2}^{d_2}Fe_{i0}(1))\in e_{00}(\sum\limits_{i=1}^{d_2}Fa_i)\subseteq e_{00}(W_{d_2}).\]
Now, $e_{00}(1)we_{00}(1)\in e_{00}(W_{d_1})e_{00}(W_{d_2})\subseteq e_{00}(W_n)$, which proves the claimed inclusion.

Let us denote the right hand side of the inclusion of Lemma \ref{Lemma3} $(2)$ as $RHS (n)$. We claim that $(Ft+Ft^{-1}+Fe_{00}(1))RHS(n-1)\subseteq RHS(n)$ and $RHS(n-1)(Ft+Ft^{-1}+Fe_{00}(1)) \subseteq RHS(n)$.

Let us check, for example, that $M_{[-n+1,n-1]\times 0}(W_i)c_{0,N}t^je_{00}(1)\subseteq RHS(n)$ provided that $i+|j|\leq n-1$. Indeed, $t^je_{00}(1)=e_{j0}(1)$,
\[c_{0N}e_{j0}(1)=\begin{cases}
0 & \text{if $j\leq0$,}\\
e_{00}(a_j) & \text{if $j\geq 1$.}
\end{cases}\]

Now,
\[M_{[-n+1, n-1]\times 0}(W_i)e_{00}(a_j)\subseteq M_{[-n+1,n-1]\times 0}(W_ia_j)\subseteq M_{[-n,n]\times 0}(W_{n-1}).\]
Hence, to check that a product of length $\leq n$ in $t^{-1},t,e_{00}(1),c_{0N}$ lies in $RHS(n)$, we may assume that the product starts and ends with $c_{0N}$. Now,
\[c_{0N}V^{n-2}c_{0N}=e_{00}(1)c_{0N}V^{n-2}c_{00}(1)c_{0N}\subseteq e_{00}(1)V^{n-1}e_{00}(1)c_{0N}\]
\[\subseteq e_{00}(W_{n-1})c_{0N}\subseteq RHS(n),\]
which completes the proof of the lemma.
\end{proof}

Denote $w(n)=\dim_F W_n$.

\begin{corollary}\label{Corollary1}
$w(n)\leq g(2n+1)$, $g(n)\leq 2(2n+1)^2w(n)+2n+1.$
\end{corollary}

\section{Growth of the Algebras $A^{(c)}$}

Now we are ready to prove Theorem \ref{Theorem1}. Let $f(n)$ be an increasing function, i.e., $f(n)\leq f(n+1)$ for all $n$ and $f(n) \rightarrow \infty$ as $n\rightarrow \infty$. Let $A$ be a countable dimensional algebra whose growth is locally weakly bounded by $f(n)$. Let $h(n)$ be a superlinear function.

Let elements $b_1, b_2, \cdots$ generate the algebra $A$. Choose a sequence $\epsilon_k>0$ such that $\lim\limits_{k\rightarrow\infty}\epsilon_k=0$. Denote $V_k=span_F(b_1, \cdots, b_k)$. By the assumption, there exist constants $c_k\geq 1$, $k\geq 1$, such that
\[dim_FV_k^n\leq c_kf(c_kn)(c_kn)^{\epsilon_k}\]
for all $n\geq 1$.

Increasing $\epsilon_k$ and $c_k$ we can assume that
\begin{equation}\label{Equation1}
dim_FV_k^n \leq f(c_kn)n^{\epsilon_k}
\end{equation}

Indeed, choose a sequence $\epsilon_k'$, $k\geq 1$, such that $0<\epsilon_k<\epsilon_k'$, \linebreak $\lim\limits_{k\rightarrow\infty}\epsilon_k'=0$.

There exists $\mu_k\geq1$ such that
\[n^{\epsilon_k'-\epsilon_k}>c_k^{\epsilon_k+1}\]
for all $n>\mu_k$.

The function $f(n)$ is an increasing function. Hence, there exists $c_k'$ such that
\[f(c_k')\geq c_kf(c_ki)(c_ki)^{\epsilon_k},\]
$i=1,\cdots, \mu_k$.

Now we have
\[dim_FV_k^n\leq f(c_k'n)n^{\epsilon_k'}\]
for all $n\geq 1$.

From now on, we will assume (\ref{Equation1}) for arbitrary $k\geq1$, $n\geq 1$.

Choose an increasing sequence $n_1<n_2<\cdots$ such that $c_kn\leq h(n)$ for all $n\geq n_k$.

Define a generating sequence $c=(a_1, a_2, \cdots)$ as follows: $a_i=b_k$ if $i=n_k$; $a_i=0$ if $i$ does not belong to the sequence $n_1, n_2, \cdots$.

We will show that the growth function of $A^{(c)}$ is weakly bounded by $f(h(n))n^2$. Choose $\alpha>0$.

For an integer $n\geq n_1$, fix $k$ such that $n_k\leq n < n_{k+1}$. Then
\[W_n=span(a_{i_1}\cdots a_{i_r}|i_1+\cdots+i_r\leq n; a_{i_1},\cdots, a_{i_r} \in \{b_1, \cdots, b_k\})\subseteq V_k^n\]
Hence, $w(n)\leq f(c_kn)n^{\epsilon_k}$. From $n_k\leq n$ it follows that $c_kn\leq h(n)$. If $n$ is sufficiently large, then we also have $\epsilon_k<\alpha$. Then
\[w(n)\leq f(c_kn)n^{\epsilon_k}\leq f(h(n))n^\alpha.\]
By Lemma \ref{Lemma3} (\ref{L3I2}) we have $g(n)\leq w(n)n^2$. Therefore $g(n) \leq f(h(n))n^{\alpha+2}$.

We have $M_\infty$-embedded the algebra $A$ as a left ideal in a finitely generated algebra $B=A^{(c)}$ of growth $\leq f(h(n))n^{\alpha+2}$.

V. Markov \cite{8} showed that for a sufficiently large $n$, the matrix algebra $M_n(B)$ is 2-generated. Clearly, $M_n(B)$ has the same growth as $B$. Since $M_n(M_\infty(A)) \cong M_\infty(A)$, it follows that the algebra $A$ is $M_\infty$-embedded in $M_n(B)$ as a left ideal. This completes the proof of Theorem \ref{Theorem1}.

In order to prove Theorem \ref{Theorem2}, we will need two elementary lemmas.

\begin{lemma}\label{Lemma4}
Let $g_k(n)$, $k\geq 1$, be an increasing sequence of subexponential functions $g_k:N\rightarrow N$, $g_k(n)\leq g_{k+1}(n)$ for all $k,n$. Then there exists a subexponential function $f:N\rightarrow N$ and a sequence $1\leq n_1 < n_2 < \cdots$, such that $g_k(n)\leq f(n)$ for all $n\geq n_k$.
\end{lemma}
\begin{proof}
Choose $k\geq 1$. From $\lim\limits_{n\rightarrow\infty}\dfrac{g_k(n)}{e^{\frac{1}{k}n}}=0$, it follows that there exists $n_k$ such that $\dfrac{g_k(n)}{e^{\frac{1}{k}n}}\leq \dfrac{1}{k}$ for all $n\geq n_k$. Without loss of generality, we will assume that $n_1<n_2<\cdots$. For an integer $n\geq n_1$, let $n_k\leq n<n_{k+1}$. Define $f(n)=g_k(n)$.

We claim that $f(n)$ is a subexponential function. Indeed, let $s\geq 1$. Our aim is to show that $\lim\limits_{n\rightarrow \infty}\dfrac{f(n)}{e^{\frac{1}{s}n}}=0$.

Let $n\geq n_s$. Let $k$ be a maximal integer such that $n_k\geq n$, so $n_k\geq n<n_{k+1}$, $s\leq k$. We have
\[\dfrac{f(n)}{e^{\frac{1}{s}n}}=\dfrac{g_k(n)}{e^{\frac{1}{s}n}}\leq \dfrac{g_k(n)}{e^{\frac{1}{k}n}}\leq\dfrac{1}{k}.\]
This implies $\lim\limits_{n\rightarrow\infty}\dfrac{f(n)}{e^{\frac{1}{s}n}}=0$ as claimed. Choose $\ell \geq 1$. For all $n\geq n_\ell$, we have $n_k\leq n <n_{k+1}$, where $\ell \leq k$. Hence, $g_\ell(n)\leq g_k(n)=f(n)$. This completes the proof of the lemma.
\end{proof}

\begin{lemma}\label{Lemma5}
Let $f(n)$ be a subexponential function. Then there exists a superlinear function $h(n)$ such that $f(h(n))$ is still subexponential.
\end{lemma}
\begin{proof}
For an arbitrary $k\geq 1$, we have $\lim\limits_{n\rightarrow\infty}\dfrac{f(kn)}{e^{\frac{1}{k^2}kn}}=0$. Hence there exists an increasing sequence $n_1<n_2<\cdots$ such that $f(kn)<\frac{1}{k}e^\frac{n}{k}$ for all $n\geq n_k$.

For an arbitrary $n\geq n_1$, choose $k\geq 1$ such that $n_k\leq n < n_{k+1}$. Let $\mu(n)=k$. Then $h(n)=n\mu(n)$ is a superlinear function since $\mu(n)\leq\mu(n+1)$ and $\mu(n)\rightarrow\infty$ as $n\rightarrow\infty$. Choose $\alpha>0$. For a sufficiently large $n$, we have $k=\mu(n)>\frac{1}{\alpha}$. Then
\[f(n\mu(n))=f(kn)<\frac{1}{k}e^{\frac{1}{k}n}<\frac{1}{k}e^{\alpha n}.\]
Hence $\lim\limits_{n\rightarrow\infty}\dfrac{f(h(n))}{e^{\alpha n}}=0$, which completes the proof of the lemma.
\end{proof}

\begin{proof}[Proof of \Cref{Theorem2}]
Let $A$ be a countable dimensional associative algebra that is locally of subexponential growth. By Lemma \ref{Lemma4}, there exists a subexponential function $f(n)$ such that the growth of $A$ is locally asymptotically bounded by $f(n)$. By Lemma \ref{Lemma5}, there exists a superlinear function $h(n)$ such that $f(h(n))$ is still a subexponential function. By Theorem \ref{Theorem1} for an arbitrary $\alpha>0$, we can $M_\infty$-embed the algebra $A$ as a left ideal in a 2-generated algebra of growth $\leq f(h(n))n^\alpha$. A product of two subexponential functions is a subexponential function. Hence, the function $f(h(n))n^\alpha$ is subexponential. This finishes the proof of Theorem \ref{Theorem2}.
\end{proof}

\begin{proof}[Proof of \Cref{Theorem3}]
Let $A$ be a countable dimensional associative algebra of Gelfand-Kirillov dimension $d$. Then the growth of $A$ is weakly asymptotically bounded by $n^d$. The function $h(n)=n\ln n$ is superlinear. By Theorem \ref{Theorem1}, the algebra $A$ is $M_\infty$-embeddable as a left ideal is a 2-generated algebra $B$ whose growth is weakly asymptotically bounded by $(n\ln n)^dn^2$, in other words, the growth of $B$ is asymptotically bounded by $n^{d+2+\alpha}(\ln n)^d$ for any $\alpha>0$. This implies $GK\dim B\leq d+2$ and completes the proof of Theorem \ref{Theorem3}.
\end{proof}

Now let us discuss the similar theorems for semigroups: \linebreak Theorems \ref{Theorem1'}, \ref{Theorem2'}, \ref{Theorem3'}.

Let $P$ be a semigroup with $1$. Let $F$ be an arbitrary field. Consider the semigroup algebra $F[P] \wr F[t^{-1},t]$. Let $c=(a_1, a_2, \cdots)$ be a sequence of elements $a_i\in P \cup \{0\}$ that generate the semigroup $P\cup\{0\}$.

Consider the algebra $F[P]^{(c)}$ and the semigroup $P^{(c)}$ generated by $t,t^{-1},e_{11}(1), c_{0,N}$. Arguing as in the proof of Lemma \ref{Lemma1}, we see that $M_\infty(P)$ is a left ideal of the semigroup $P^{(c)}$.

Starting with an arbitrary generating sequence $b_1, b_2, \cdots$ of the semigroup $P$ and diluting it with zeros as in the proof of Theorem \ref{Theorem1}, we get a generating sequence $c=(a_1, a_2, \cdots)$ of the semigroup $P\cup\{0\}$ such that the semigroup $P^{(c)}$ has the needed growth properties. The proof just follow from the proofs of Theorem \ref{Theorem1}, \ref{Theorem2}, \ref{Theorem3}.

\bibliographystyle{amsplain}
\bibliography{ASLSGbib}

\section*{Acknowledgement}

The project of the first two authors was funded by the Deanship of Scientific Research (DSR), King Abdulaziz University. The authors, therefore, acknowledge technical and financial support of KAU.

The fourth author gratefully acknowledges the support from the NSF.

\endnotetext[1]{Department of Mathematics, King Abdulaziz University, Jeddah, SA,\\
E-mail address, ANALAHMADI@KAU.EDU.SA; HHAALSALMI@KAU.EDU.SA;}

\endnotetext[2]{Department of Mathematics, Ohio University, Athens, USA,\\
E-mail address, JAIN@OHIO.EDU;}

\endnotetext[3]{Department of Mathematics, University of California, San Diego, USA\\
E-mail address, EZELMANO@MATH.UCSD.EDU}

\theendnotes

\end{document}